 0
 0

\newcommand{\R} [1] { {\bf{R}}^{#1} }
\newcommand{\T} [1] { {\bf{T}}^{#1} }
\newcommand{\Sphere} [1] { {\bf{S}}^{#1} }
\newcommand{\SO} [1] { {\rm{SO}}_{#1}\R{} }
\newcommand{\Disc} [1] {{\bf D}^{#1} }
\newcommand{\Z} [1] { {\bf{Z}}^{#1} }

\newcommand{\C} [1] { {\bf{C}}^{#1} }

\newcommand \im { {\rm im}\, }

\def\D{{\mathcal D}}

\def\n{{\mathfrak n}}

\documentclass[11pt,a4paper,reqno]{amsart}
\usepackage{amssymb,amsmath,amsthm}
\usepackage[all]{xy}
\xyoption{import}

\numberwithin{equation}{section}
\newtheorem{thm}{Theorem}[section]  
\newtheorem{lem}[thm]{Lemma}        
\theoremstyle{definition}

\begin{document}

\title{Geometry and real-analytic integrability \\
}

\author[L. T. Butler]{Leo T. Butler}
\address{School of Mathematics\\
The University of Edinburgh,
6214 James Clerk Maxwell Building,
Edinburgh, UK, EH9 3JZ}
\email{l.butler@ed.ac.uk}
\thanks{The author thanks Alexei Bolsinov and two anonymous referees
for their comments.}

\begin{abstract}
This note constructs a compact, real-analytic, riemannian $4$-manifold
$(\Sigma,{\bf g})$ with the properties that: (1) its geodesic flow is
completely integrable with smooth but not real-analytic integrals; (2)
$\Sigma$ is diffeomorphic to $\T{2} \times \Sphere{2}$; and (3) the
limit set of the geodesic flow on the universal cover is dense. This shows
there are obstructions to real-analytic integrability beyond
the topology of the configuration space.
\end{abstract}

\date{\today}

\maketitle

\section{Introduction}

Ta\u{\i}manov proves that if $(\Sigma,{\bf g})$ is a compact
real-analytic manifold whose geodesic flow is integrable with
real-analytic first integrals, then there is an invariant torus
$\T{}$ in the unit-tangent sphere bundle $S\Sigma$ such that
$\pi_1(\T{}) \to \pi_1(\Sigma)$ is almost surjective~\cite{Taimanov:1988a}. The example
in~\cite{Butler:2000a} shows that this is false in the smooth
category.  In that example $(\Sigma,{\bf g})$ is a compact
real-analytic riemannian $3$-manifold with a nilpotent $\pi_1$ that is not almost abelian.

To state the present note's results: Let $E$ be the total space of a
$\C{1}$-vector bundle over $\T{2}$ with even Euler number. The inclusion
of $\C{}$ into $\Sphere{2} = \C{} \cup
\{\infty\}$ induces a compactification of $E$ into an $\Sphere{2}$-bundle
$\Sigma$ over $\T{2}$
with a natural action of $\SO{2}$ on
its fibres. Define a
metric ${\bf g}$ on $\Sigma$ by $\SO{2}$-equivariantly identifying
 $T\Sigma$ as the orthogonal direct
sum of the sub-bundle $V\simeq T\Sphere{2}$ of vertical fibres and a
horizontal sub-bundle $H \simeq T\T{2}$ and equip each
fibre of $V$ (resp. $H$) with the
metric of the round sphere of unit radius (resp. standard flat metric).
There are natural identifications of $H$ and $V$ that allow ${\bf g}$ to be
uniquely defined (see sections 3.3--3.4).

\begin{thm} \label{thm:a}
The compact, real-analytic, riemannian $4$-manifold
$(\Sigma,{\bf g})$ has the following properties
\begin{enumerate}
\item its geodesic flow, $\varphi : \R{} \times S\Sigma \to S\Sigma$,
is completely (resp. non-commutatively) integrable with smooth but not
real-analytic integrals;
\item $\Sigma$ is bundle isomorphic to $\T{2} \times \Sphere{2}$;
\item the limit set of its geodesic flow on the universal cover is dense; and
\item its geodesic flow has zero topological entropy.
\end{enumerate}
\end{thm}

Since the geodesic flow of $(\T{2} \times \Sphere{2}, {\bf g}_{\rm
flat} \times {\bf g}_{\rm round})$ is real-analytically integrable,
the topology of $\Sigma$ does not preclude real-analytically
integrable geodesic flows. How does one know that the first integrals in
Theorem~\ref{thm:a} cannot be real analytic? It is shown that if $\T{} \subset S\Sigma$ is a regular
$\varphi$-invariant torus, then the subgroup ${\rm im\ }(\pi_1(\T{})
\to \pi_1(\Sigma))$ has rank at most $1$. By Ta\u{\i}manov's result,
mentioned above, this precludes real-analytic integrability (see sections 3.5--3.6).

The geodesic flow $\varphi$ is more than just smoothly integrable. Let
us recall some terminology to describe a type
of topologically-tame integrability~\cite{Butler:2005a}: A $C^r$
($1\leq r\leq \infty$) action $\phi : \R{s} \times M \to M$ is {\em
integrable} if there is an open, dense subset $R$ that is covered by
angle-action charts $(\theta,I) : U \to \T{k} \times \R{l}$ which
conjugate $\phi_t$ ($t \in \R{s}$) with a translation-type map
$(\theta,I) \mapsto (\theta + \omega(I)t, I)$ where $\omega :
\R{l} \to {\rm Hom(}\R{s},\R{k}{\rm )}$ is a smooth map. Evidently, there
is an open dense subset $L \subset R$ fibred by $\phi$-invariant
tori~\cite{Bogo:1997a}. Let $f : L \to B$ be the $C^r$ fibration which
quotients $L$ by these invariant tori and let $\Gamma = M - L$ be the
{\em singular set}. If $\Gamma$ is a tamely-embedded polyhedron, then
$\phi$ is called $k$-{\em semisimple} with respect to
$(f,L,B)$.

\begin{thm} \label{thm:b}
The geodesic flow $\varphi$ is $4$-semisimple with respect to an
$(f,L,B)$ such that $f$ has non-trivial monodromy and a trivial
Chern class; it is also $3$-semisimple with respect to an $(f',L',B')$
such that $f'$ has trivial monodromy and a non-trivial Chern
class. The fibres of $f'$ are contractible in $S\Sigma$.
\end{thm}

\noindent {\em Remarks.} (1) The contractibility of the fibres of
$f'$ and the minimality of $\varphi$ restricted to a generic fibre
of $f'$ implies that $\varphi$ is not real-analytically integrable
(see sections 3.5--3.6). The contractibility of the fibres of $f'$
also implies the density of the geodesic flow's limit set on the
universal cover. (2) If $\Sigma$ is constructed from $E$ with an odd
Euler number, then parts 1,3,4 of Theorem~\ref{thm:a} and
Theorem~\ref{thm:b} are true. The cohomology and homotopy rings of
$\Sigma$ are isomorphic to those of $\T{2} \times \Sphere{2}$, but
$\Sigma$ is not isomorphic to $\T{2} \times \Sphere{2}$. (3) The
example of Theorem~\ref{thm:a} generalizes via the non-degenerate
$2$-step nilmanifolds of~\cite{Butler:2003a}. These examples are
multiply-connected. {\em Are there simply-connected examples?} (4)
Theorem~\ref{thm:a} arose from an attempt to understand whether a
hyperbolic manifold admits a semisimple geodesic flow. If so, then
its invariant tori are contractible and property 3 of
Theorem~\ref{thm:a} is true --- morally, at least. See Theorems 6--7
and Question C of~\cite{Butler:2005a} for precise statements. The
examples of the present note suggest that there may exist a
semisimple geodesic flow on a hyperbolic manifold.

\section{The topology of $\Sigma$}

{\em Terminology}: Let $G$ be a group. The {\rm anti-diagonal}
subgroup is the subgroup $\{ (g,g^{-1})\ :\ g \in G\}$ of $G \times
G$. If $X$ and $Y$ are $G$-spaces, then define $X \times_G Y$ to be
the set of orbits of the anti-diagonal subgroup.
\medskip

Since $\Sigma$ is an $\Sphere{2}$-bundle over $\T{2}$, there is a
canonical principal $\SO{3}$-bundle $\SO{3} \hookrightarrow P \to
\T{2}$ such that $\Sigma = P \times_{\SO{3}} \Sphere{2}$ is an
associated bundle.

\begin{lem} \label{lem:trivial}
The Euler number of $E$ is even iff $P$ (hence $\Sigma$) is trivial.
\end{lem}

\begin{proof} It suffices to show that $P$ admits a section under the
stated condition. To do this, let us recall some obstruction theory. Let
$G$ be a compact, connected Lie group and let $G \hookrightarrow B \to
X$ be a principal $G$-bundle over a CW complex $X$. Let us
try to construct a section of $B$. Let $X_p$ be the $p$-skeleton of
$X$ and let $\D_p$ be $X$'s cellular chain group in dimension
$p$. Assume that $s_p : X_p \to B|X_p$ is a section of $B|X_p$. The
obstruction to extending $s_p$ to a section $s_{p+1}$ over $X_{p+1}$
is a cochain $\theta \in {\rm Hom}(\D_{p+1};
\pi_p(G))$ -- the coefficients are untwisted because $B$ is a
principal fibre bundle.

The cochain $\theta$ can be defined on the
generators of $\D_{p+1}$ as follows. Let $f_i : (\Disc{p+1},
\Sphere{p}) \to (X_{p+1},X_p)$ be the attaching map of a $p$-cell and
$\tau_i = f_{i,*} [\Disc{p+1},
\Sphere{p}]$ be the induced cellular chain (orient $\Disc{p+1}$
arbitrarily). The pullback bundle $f_i^*B$ is trivial since
$\Disc{p+1}$ is contractible, so let $\mu_i : f_i^*B \to G$ be a
principal fibre-bundle map. Define $\langle \theta, \tau_i \rangle
\in \pi_p(G)$ to be the homotopy class of the map
$$
\xymatrix{
\Sphere{p} \ar[r]^{f_i} & X_p \ar[r]^{s_p} & B|X_p \ar[r]^{\rm{incl.}} &
 f_i^*B \ar[r]^{\mu_i} & G.
}
$$
The obstruction cochain $\theta$
measures the compatibility of the trivialization of $B|X_p$ induced by
$s_p$ and the trivialization of $f_i^*B$ induced by the contractibility
of $\Disc{p+1}$.

$B|X_1$ admits a section $s_1$ since $\pi_1(B)$ acts trivially on
$\pi_*(G)$. Assume that $X_2$ contains a single $2$-cell with
attaching map $f$. The obstruction cochain is the homotopy class of
$$
\xymatrix{
\Sphere{1} \ar[r]^{f} & X_1 \ar[r]^{s_1} & B|X_1 \ar[r]^{\rm{incl.}} &
 f^*B \ar[r]^{\mu} & G.
}
$$

Let $E_1 \subset E$ be the unit-circle bundle, which is also a
principal $\SO{2}$-bundle. When $(G,B,X)=(\SO{2},E_1,\T{2})$, the
above discussion shows that the obstruction cochain $\theta$ lies in
${\rm Hom}(\Z{}; \pi_1(\SO{2})) = \Z{}$ and it can be identified with
the Euler class of $E$. When $(G,B,X) = (\SO{3},P,\T{2})$, the
obstruction cochain $\theta'$ lies in ${\rm Hom}(\Z{};
\pi_1(\SO{3})) = \Z{}_2.$ Moreover, it is clear that the map
$\pi_1(\SO{2}) \to \pi_1(\SO{3})$ maps $\theta \to \theta'$, i.e. the
reduction of $\theta$ mod $2$ is $\theta'$. This proves the lemma.
\end{proof}

\section{The metrics}

The  metric ${\bf g}$ is constructed from an $\SO{2} \times \SO{2}$-invariant metric
on $E_1 \times \Sphere{2}$ using the representation
$$\Sigma = E_1~\times_{\SO{2}}~\Sphere{2}.$$ We
construct an $\SO{2}$-invariant metric on $E_1$, first.

\subsection{$E_1$} \label{ssec:E1}

The unit-circle bundle $E_1$ is diffeomorphic to the nilmanifold
$$\Gamma \backslash N,$$
where $N=Nil$ is the unipotent group $3 \times 3$ upper triangular
matrices and $$\Gamma = \left\langle \alpha = \left( \begin{array}{ccc} 1 & 1 &
0\\0 & 1 & 0\\ 0 & 0 & 1 \end{array} \right),\ \beta = \left( \begin{array}{ccc} 1 & 0 &
0\\0 & 1 & 1\\ 0 & 0 & 1 \end{array} \right),\ \gamma = \left(
\begin{array}{ccc} 1 & 0 & \frac{1}{k}\\0 & 1 & 0\\ 0 & 0 & 1 \end{array}
\right)\ \right\rangle,$$
where $k\neq 0$ is the Euler number of $E_1$. Define a riemmannian
metric ${\bf g}_1$ on $E_1$ by declaring that $\log \alpha$, $\log
\beta$, $\log \gamma$ is an orthonormal, left-invariant frame on $N$. Let $x : \Gamma
\backslash N \to \Sphere{1}$ (resp. $y : \Gamma
\backslash N \to \Sphere{1}$) be the map induced by $g
\in N
\mapsto \langle \log g, \log \alpha \rangle$ (resp. by $\beta$). The
map $(x,y) : E_1 \to \T{2}$ is a non-canonical form of the fibre
bundle map $E_1 \to \T{2}$.

\def\d{{\rm d}}
The metric ${\bf g}_1$ has the explicit form at $\Gamma g$
$${\bf g}_1 = \d x^2 + \d y^2 + (\d z-\frac{1}{2}(x\d y-y\d x))^2,$$
where $\log g = x \log \alpha + y \log \beta + z \log \gamma$.

The left-trivialization of $T^*N$ induces a trivialization of $T^*E_1
= \Gamma \backslash N \times \n^*$, where $\n^*$ is the dual space of
$N$'s Lie algebra. Let $a : \n^* \to \R{}$ (resp. $b,c$) be the linear
form induced by $\log \alpha$ (resp. $\log \beta$, $\log \gamma$). The
hamiltonian of ${\bf g}_1$ is
$$H_1 = \frac{1}{2}(a^2 + b^2 + c^2),$$
which has smooth first integrals
$$f_1 = c,\ \ f_2 = \exp(-c^{-2}) \sin 2\pi\left(\frac{a}{c} + y \right) , \ \
f_3 = \exp(-c^{-2}) \sin 2\pi\left(\frac{b}{c} - x \right) .$$

The centre of $N$ is $Z(N) = \exp(\R{} \log \gamma)$ while the centre
of $\Gamma$ is $Z(\Gamma) = \exp(\Z{} \log \gamma)$. $Z(N)/Z(\Gamma)$
acts freely on $\Gamma\backslash N$. From the identification $\SO{2} =
Z(N)/Z(\Gamma)$, we see that the fibres of $E_1$ are just the orbits
of $Z(N)/Z(\Gamma)$. Moreover, the momentum map of this
$\SO{2}$-action on $T^* E_1$ is
$$\Psi_1(\Gamma g, \mu) = c(\mu),$$
for all $g \in N$, $\mu \in \n^*$.

\subsection{$\Sphere{2}$} \label{ssec:S2}
\def\r{{\tt r}}

Let $\Sphere{2} = \{ \xi \in \R{3}\ : \ |\xi|=1 \}$ and let $T^*\Sphere{2}
= \{ (\xi,p) \in \R{3} \times \R{3}\ : \ |\xi|=1, \langle \xi,p \rangle=0
\}$ where $\langle,\rangle$ is the euclidean metric on $\R{3}$. Let
${\bf g}_2$ be the round metric on $\Sphere{2}$ induced by the
euclidean metric. Let $\mu_i$ be the standard orthonormal basis of $so_3\R{}$ and
define the 1-form $\alpha_i$ at $\xi$ to be $\alpha_i(\bullet) = \langle
\bullet, \mu_i \xi\rangle$. The round metric is expressed as
$${\bf g}_2 = \alpha_1^2 + \alpha_2^2 + \alpha_3^3 = \d \r^2 + 4 \pi^2 \sin^2 \r\ \d \phi^2,$$
where $(\r,\phi \bmod 1)$ are radial coordinates in which the vector field $\mu_1\xi$ equals $\frac{1}{2\pi} \frac{\partial\ }{\partial \phi}$.

The momentum map of the $\SO{2}$ action on $T^*
\Sphere{2}$ is
$$\Psi_2(\xi,p) = \langle p, \mu_1 \xi \rangle = \frac{1}{2\pi} p_{\phi},$$ while the hamiltonian of
the metric ${\bf g}_2$ is
$$H_2 = \frac{1}{2}\left(  \langle p, \mu_1 x \rangle^2 +  \langle p, \mu_2 x
\rangle^2 +  \langle p, \mu_3 x \rangle^2 \right) = \frac{1}{2}\left( p_{\r}^2 + (2\pi\sin\r)^{-2} \ p_{\phi}^2 \right).$$

\subsection{$E_1 \times \Sphere{2}$ and $\Sigma$} \label{ssec:sigma}

The riemannian metric ${\bf g}_1 \times {\bf g}_2$ has hamiltonian
$H_1+H_2$ and first integrals $f_1,f_2,f_3$ \& $\Psi_1,\Psi_2$
[N.B. $f_1=\Psi_1$]. The metric is also $\SO{2} \times
\SO{2}$-invariant, hence it is invariant under the anti-diagonal
action of $\SO{2}$. Therefore, there is a well-defined submersion
metric ${\bf g}$ on $\Sigma = E_1
\times_{\SO{2}} \Sphere{2}$. The momentum map of the anti-diagonal
$\SO{2}$-action is $\Psi = \Psi_1 - \Psi_2$.

It is well-known that $T^*\Sigma$ is symplectomorphic to
$\Psi^{-1}(0)/\SO{2}$. It is trivial to verify that

\begin{lem}
$H_1, H_2, f_1, f_2, f_3$ are functionally independent on an open
dense subset of $\Psi^{-1}(0)$. Moreover, $H_1,H_2,f_1,f_2$ Poisson
commute while $H_1, H_2, f_1$ Poisson commute with $f_2, f_3$.
\end{lem}

This Lemma proves that the geodesic flow of $(\Sigma,{\bf g})$ is
completely integrable with the integrals induced by $H_1,H_2,f_1,f_2$;
and it is non-commutatively integrable with the integrals induced by
$H_1,H_2,f_1,f_2,f_3$. Note that $f_2$ and $f_3$ are only smooth.

\subsection{An explicit expression for ${\bf g}$} \label{ssec:g}

\def\s{{\tt s}}
\def\t{{\tt t}}

Let $\hat{E}_1 = N/Z(\Gamma)$,
the universal abelian covering space of $E_1$, and let
$\hat{\Sigma} = \hat{E}_1 \times_{\SO{2}} \Sphere{2}$ be the universal covering
space of $\Sigma$. There is the commutative diagram of riemannian
im/submersions
$$
\xymatrix{
(\R{2} \times \Sphere{1} \times \Sphere{2}, \hat{\bf K}) \ar@{..>}[d]^{\bar{\rho}} \ar@{..>}[r]^{W} & (\hat{E}_1 \times \Sphere{2}, \hat{{\bf g}}_1 + {\bf g}_2) \ar[r]\ar[d] & (\hat{E}_1 \times \Sphere{2}, {\bf g}_1 + {\bf g}_2) \ar[d]\\
(\R{2} \times \Sphere{2}, \hat{\bf k}) \ar@{..>}[r]^{w} & (\hat{\Sigma},\hat{{\bf g}}) \ar[r] & (\Sigma,{\bf g}),
}
$$
where dotted arrows indicate maps that remain to be defined.
An explicit formula for ${\bf g}$ in local coordinates is provided by
$\hat{\bf k}$.

Let $(x,y,z \bmod 1)$ be coordinates on $\hat{E}_1 = N/Z(\Gamma)$
and let $(\r,\phi \bmod 1)$ be polar coordinates on $\Sphere{2}$
(section~\ref{ssec:E1}--\ref{ssec:S2}). Identify $\SO{2}$ with
$\R{}/\Z{}$. The anti-diagonal action of $\SO{2}$ on $\hat{E}_1
\times \Sphere{2}$ is
$$\theta*(x,y,z \bmod 1,\r,\phi \bmod 1) = (x,y,z+\theta \bmod 1, \r, \phi-\theta \bmod 1),$$
for all $\theta$ in $\SO{2}$. Let $\s = z+\phi$ and $\t = z$, so that
$(x,y,\r,\s,\t)$ is a system of local coordinates on $\hat{E}_1 \times \Sphere{2}$.
This defines $W$; since the
$\SO{2}$-orbits are the circles $(x,y,\r,\s)=constant$, this also defines
$\bar{\rho}$ and $w$.
Since $\hat{\bf K} = W^*(\hat{\bf g}_1 + {\bf g}_2)$,
\begin{eqnarray*}
\hat{\bf K} &=&
\d x^2 + \d y^2 + (\d\t - \frac{1}{2}(x\d y - y\d x))^2 +\\
&& \d\r^2 + 4\pi^2 \sin^2 \r\, ( \d \s - \d\t )^2.
\end{eqnarray*}
The frame $\left\{ \frac{\partial\ }{\partial x}-\frac{1}{2}y \frac{\partial\ }{\partial \s} , \frac{\partial\ }{\partial y}+ \frac{1}{2}x\frac{\partial\ }{\partial \s} ,  \frac{\partial\ }{\partial \r},  \frac{\partial\ }{\partial \s} \right\}$ on
$\R{2} \times \Sphere{2}$ horizontally lifts to the frame
{
\def\X{{\mathcal X}}
\def\Y{{\mathcal Y}}
\def\R{{\mathcal R}}
\def\S{{\mathcal S}}
\def\T{{\mathcal T}}
$\left\{
\X =  \frac{\partial\ }{\partial x} - \frac{1}{2}y\left(  \frac{\partial\ }{\partial \s} +  \frac{\partial\ }{\partial \t} \right), \right.$ $
\Y = \frac{\partial\ }{\partial y} + \frac{1}{2}x\left(  \frac{\partial\ }{\partial \s} +  \frac{\partial\ }{\partial \t} \right), $ $
\R = \frac{\partial\ }{\partial \r} , $ $
\left.
\S = \frac{\partial\ }{\partial \s} + \frac{4\pi^2\sin^2 \r}{1+4\pi^2\sin^2 \r}
\frac{\partial\ }{\partial \t}
\right\}
$ on ${\bf R}^2 \times \Sphere{1} \times \Sphere{2}$.
Along with $\T = \frac{\partial\ }{\partial \t}$, this forms a $\hat{\bf K}$-orthogonal frame. A simple calculation shows that $\X,\Y$ and $\R$ have unit norm
and the norm of $\S$ is $\sqrt{\frac{4\pi^2\sin^2 \r}{1+4\pi^2\sin^2 \r} }$. Therefore
$$\hat{\bf k} = \d x^2 + \d y^2 + \d \r^2 + \frac{4\pi^2\sin^2 \r}{1+4\pi^2\sin^2 \r}\, (\d \s - \frac{1}{2}(y\d x - x\d y))^2.$$
}
\medskip
\noindent{\em Remarks}. (1) Clearly $\hat{\bf
k}|_{\{x=const.,y=const.\}}$ equals $\d \r^2 + \frac{4\pi^2\sin^2
\r}{1+4\pi^2\sin^2 \r}\, \d \s^2$, which is a non-round
$\SO{2}$-invariant metric on $\Sphere{2}$. In addition, it is clear
that these fibres are totally geodesic. (2) If $u$ is a smooth
function which vanishes at $\r =0, \pi$, has $u'(0)=2\pi$,
$u'(\pi)=-2\pi$ and the even derivatives of $u$ vanish at $\r =0,
\pi$, then
$$\hat{\bf k}_u = \d x^2 + \d y^2 + \d \r^2 + u(\r)^2 \, (\d\s - \frac{1}{2}(y\d x - x\d y))^2$$
defines a smooth metric on $\R{2}\times \Sphere{2}$. A simple computation shows that
$\hat{\bf k}_u$ is invariant under the deck-transformation group and so induces
a metric ${\bf g}_u$ on $\Sigma$. Theorem 1.1 is true for $(\Sigma,{\bf g}_u)$, and
if $u$ is analytic, then Theorem 1.2 is true also.
If $u(\r)=2\pi \sin \r$, then
$\{x=const.,y=const.\}$ is a totally geodesic round $\Sphere{2}$
but $(\Sigma, {\bf g}_u)$ is not isometric to $(\T{2} \times \Sphere{2},
{\bf g}_{{\rm flat}} + {\bf g}_{{\rm round}})$.

\subsection{The limit set}

Say that a point is recurrent for a flow if its orbit visits
arbitrarily small neighbourhoods of itself in both forward and
backward time. The limit set is the closure of the set of
recurrent points.

Let $(\tilde{\Sigma}, \tilde{\bf g} )$ be the universal riemannian
cover of $(\Sigma,{\bf g})$ and let $\tilde{\varphi}$ be the geodesic
flow of $(\tilde{\Sigma}, \tilde{\bf g} )$. The functions $x,y,a,b,c$
induce well-defined smooth functions on $T^*\Sigma$ and hence on
$T^*\tilde{\Sigma}$. Let $\tilde{x}$ ($\tilde{y}$) be the
single-valued function on $T^*\tilde{\Sigma}$ induced by $x$
($y$). The map
$$\nu = ( a + c\tilde{y}, b-c\tilde{x}, c ), \qquad \nu :
T^*\tilde{\Sigma} \to \R{3}$$ is a first integral of $\tilde{\varphi}$
(it is a coordinatized incarnation of the momentum map of $N$'s left
action on $T^*\tilde{\Sigma}$). For $\nu_o=(\nu_1,\nu_2,\nu_3) \in
\R{3}$ such that $\nu_3
\neq 0$, it is apparent that if $\nu(P) = \nu_o$, then
$$|\tilde{x}(P)| \leq \left|\frac{b(P)+\nu_2}{\nu_3}\right|, \qquad |\tilde{y}(P)| \leq
\left| \frac{a(P)+\nu_1}{\nu_3}\right|. \eqno (*)$$
On $S\tilde{\Sigma}$ the functions $a$ and $b$ are bounded
above by unity, so
\begin{lem}
The map $\nu | \{c \neq 0\} \cap S\tilde{\Sigma}$ is proper.
\end{lem}

Let ${\bf S} \subset S\Sigma$ be a regular invariant $3$-torus and let
$\tilde{\bf S} \subset S \tilde{\Sigma}$ be a lift of this torus. Since
$\tilde{\bf S}$ is regular it lies in $\{c\neq 0\}$, and therefore it
is a closed subset of a fibre of $\nu$. Hence $\tilde{\bf S}$ is
compact. But $\tilde{\varphi} | \tilde{\bf S}$ is a translation-type
flow, so its limit set is $\tilde{\bf S}$. Since the union of regular
tori is dense, this proves that the limit set is dense. Since the
limit set is closed, it is $S\tilde{\Sigma}$. This proves part 3 of
Theorem~\ref{thm:a}.

\subsection{Real-analytic non-integrability}
Let's complete the proof of part 1 of Theorem~\ref{thm:a}. For the next
seven paragraphs inclusive $(\Sigma, {\bf g})$ is a compact real-analytic
riemannian manifold with geodesic flow $\varphi$ and first Betti number q.
Assume that $\varphi$ is real-analytically integrable
(or more generally, geometrically simple~\cite{Taimanov:1988a}) and let
${\mathcal T}$ be the induced singular fibration of $S\Sigma$.
A regular fibre of ${\mathcal T}$ is an
isotropic torus.
Let the lift of $\bullet$ on $S\Sigma$
to the universal cover $S\tilde{\Sigma}$ be denoted by $\tilde{\bullet}$.

\begin{lem} \label{lem:taim}
There is an open set ${\mathcal U}$ of fibres of ${\mathcal T}$ such that
for each $\T{} \in {\mathcal U}$, $\tilde{\T{}}$ is diffeomorphic to a
cylinder $\T{{\rm r}}\times \R{{\rm q}}$ where ${\rm r}+{\rm q} \leq \dim \Sigma$.
\end{lem}

\begin{proof} Ta\u{\i}manov's Theorem says that
there is an open set ${\mathcal U} \subset {\mathcal T}$ of regular invariant tori
such that for each $\T{} \in {\mathcal U}$ the $\pi_1$-image of $\T{} \to \Sigma$ has
finite index.
Therefore, there is a splitting $\T{} = \T{{\rm r}} \times \T{{\rm q}}$ where
$\pi_1(\T{}) \to \pi_1(\Sigma)$ factors through an injection
$\pi_1(\T{{\rm q}}) \to \pi_1(\Sigma)$.
\end{proof}

\noindent
{\em Remark}. It is clear from the construction of action-angle variables
that the splitting of $\tilde{\T{}}$ can be made compatible with its
tautological affine structure.

\medskip

Let us continue with the notation and hypotheses of Lemma~\ref{lem:taim}. In
addition,
\begin{lem} \label{lem:min}
Let ${\mathcal S}$ be a singular fibration of $S\Sigma$ whose
regular fibres are ${\rm s}$-dimensional $\varphi$-invariant tori.
If there is a dense set of fibres of $\tilde{\mathcal S}$ on which
$\tilde{\varphi}$ is minimal, then ${\rm s} \leq {\rm r}$.
\end{lem}

\begin{proof} Recall that a flow is minimal if every orbit is dense.
Let $v \in S\tilde{\Sigma}$ be a regular
point of each singular fibration and let
$\tilde{\bf S} \in \tilde{\mathcal S}$ (resp.
$\tilde{\T{}} \in \tilde{\mathcal T}$)
be the regular $\tilde{\mathcal S}$-fibre (resp. $\tilde{\mathcal T}$-fibre)
through $v$. It can be assumed that $v$ is chosen so that
$\tilde{\varphi}|\tilde{\bf S}$ is minimal and
$\tilde{\T{}} \in \tilde{\mathcal U}$.
Since $\tilde{\varphi}|\tilde{\bf S}$ is minimal, $v$ is a recurrent point
for $\tilde{\varphi}|\tilde{\T{}}$. Therefore, all points of $\tilde{\T{}}$ are
recurrent. Let $\tilde{\T{}} = \T{{\rm r}} \times \R{{\rm q}}$ be a splitting from
Lemma~\ref{lem:taim} which is compatible with the tautological affine structure
on $\tilde{\T{}}$ and let $v=(v',v'')$ relative to this splitting. Since $v$
is a recurrent point for $\tilde{\varphi}|\tilde{\T{}}$, the orbit closure
 $\overline{\tilde{\varphi}_{\R{}}(v)}$ must be contained in $\T{{\rm r}} \times \{v''\}$.
Since $\tilde{\varphi}|\tilde{\bf S}$ is minimal, $\overline{\tilde{\varphi}_{\R{}}(v)} =
\tilde{\bf S}$. Therefore $\tilde{\bf S} \subset \T{{\rm r}} \times \{v''\}$. Thus
${\rm s} \leq {\rm r}$.
\end{proof}

\noindent
{\em Proof of Part 1 of Theorem~\ref{thm:a}}.
Let us return to the $(\Sigma,{\bf g})$ constructed in sections 3.1--3.4.
In this case ${\rm q}=2$ and ${\mathcal S}$ is the singular
fibration of $S\Sigma$ by the connected components of the common level sets
of $H_1,H_2,f_1,f_2$ and $f_3$, so ${\rm s}=3$.
The inequalities (*) and the subsequent argument shows that every
regular fibre of $\tilde{\mathcal S}$
is a $3$-torus. The geodesic flow $\varphi$ is minimal on a
dense set of regular fibres of ${\mathcal S}$, hence $\tilde{\mathcal S}$.
This follows from equation (9) of~\cite{Butler:2000a}, which constructs
action-angle variables for the geodesic flow of $(E_1,{\bf g}_1)$. By
Lemma~\ref{lem:taim}, if $\varphi$ is
real-analytically integrable, then ${\rm r} \geq {\rm s}=3$.
Since ${\rm r}+{\rm q} \leq 4$, ${\rm q} \leq 1$. As ${\rm q}=2$, this is
absurd. Therefore, $\varphi$ is not real-analytically integrable in either
the commutative or non-commutative sense.
This proves part 1 of Theorem~\ref{thm:a}.

\medskip

\noindent
{\em Remark}. The fact that $\varphi$ is non-commutatively integrable appears
to be important for the above proof of real-analytic non-integrability. This is
mistaken.
The hamiltonian $H'_1 = (2+\sin 2\pi y) H_1$ on $T^*E_1$ plus $H_2$ on
$T^*\Sphere{2}$ induces a real-analytic metric ${\bf g}'$ on $\Sigma$ whose
geodesic flow $\varphi'$
is completely integrable with integrals $H'_1, H_2, f_1$ and $f_2$. Construction
of action-angle variables for the induced singular toral fibration ${\mathcal S}$
shows that $\varphi'$ is minimal on a generic
fibre of $\mathcal S$. The inequality (*) for $\tilde{y}$ shows that the
regular fibres of $\tilde{\mathcal S}$ can only go to infinity in the
$\tilde{x}$-direction, so they are either compact $\T{4}$ or cylinders
$\T{3} \times \R{1}$.

On the other hand, assume that $\varphi'$ is real-analytically integrable with induced
singular fibration ${\mathcal T}$. Let $v \in S\Sigma$ be a regular point of both
singular fibrations. Let ${\bf S} \in {\mathcal S}$ and $\T{} \in {\mathcal T}$ be
the respective fibres through $v$. By the density of minimal fibres, it can be
assumed that $\varphi'|{\bf S}$ is minimal and that $\T{} \in {\mathcal U}$.
Since $\T{}$ is a $\varphi'$-invariant torus, minimality implies
${\bf S} \subset \T{}$. Since $\T{}$
is isotropic, it is a lagrangian torus of the same dimension of ${\bf S}$.
Therefore, ${\bf S}$ is open and closed in
$\T{}$, hence ${\bf S} = \T{}$. By Lemma~\ref{lem:taim}, $\tilde{\T{}}$ splits
as $\T{2} \times \R{2}$. By the above comments $\tilde{\bf S} = \T{3} \times \R{1}$
or $\T{4}$.
But $\tilde{\T{}}$ is homeomorphic to $\tilde{\bf S}$. Absurd.

\section{Semisimplicity}

Let $\Gamma'_1 = \{ p \in T^* \Sigma\ : \ c (2 H_1 - c^2)(2 H_2 - c^2) = 0
\}$, $L'_1 = T^* \Sigma - \Gamma'_1$ and $B'_1 = \T{2} \times \R{*}
\times \R{+} \times \R{+}$. Define the map $f'_1 : L'_1 \to B'_1$ by
$$f'_1 = (\frac{a}{c}+y\ \bmod 1, \frac{b}{c}-x\ \bmod 1, c, 2H_2 -
c^2, 2H_1-c^2).$$ This map is a proper submersion whose fibres are
isotropic, $\varphi$-invariant $3$-tori. The singular set $\Gamma'_1$
is a real-analytic set, hence a tamely-embedded polyhedron.

Similarly, let $\Gamma_1 = \Gamma'_1$, $L_1 = L'_1$ and $B_1 = \T{1}
\times  \R{*} \times \R{+} \times \R{+}$. Define the map $f_1 : L_1 \to B_1$ by
$$f_1 = (\frac{a}{c}+y\ \bmod 1, c, 2H_2 - c^2, 2H_1-c^2).$$ This map
is a proper submersion whose fibres are isotropic, $\varphi$-invariant
$4$-tori. These two constructions, along with the arguments
in~\cite{Butler:2005a}, imply all but the final sentence of
Theorem~\ref{thm:b}.

The map $f'_1$ extends to a map $f'_2 : L'_2 = \{ c\neq 0 \}
\to B'_2 = \T{2} \times \R{*} \times \R{\geq 0} \times \R{\geq 0}$. By
the homotopy-lifting theorem, the inclusion of a fibre of $f'_1$ is
homotopic to the inclusion of the fibre $\T{}_{s,t}$ over $(0 \bmod
\Z{2},1,s,t)$ ($s,t > 0$). As $\T{}_{0,0}$ is an elliptic critical
fibre for $f'_2$, and an $\SO{2}$-orbit, it follows that
$\T{}_{s,t}$ is contractible in $L'_2$. This suffices to complete the
proof of Theorem~\ref{thm:b}.

\bibliographystyle{plain}

\def\polhk#1{\setbox0=\hbox{#1}{\ooalign{\hidewidth
  \lower1.5ex\hbox{`}\hidewidth\crcr\unhbox0}}}

\end{document}